\documentclass[11pt,reqno]{amsart}
\usepackage{amssymb}

\usepackage{microtype}
\usepackage{stmaryrd}
\usepackage{leftidx}
\usepackage{setspace}
\usepackage[utf8]{inputenc}
\usepackage[T1]{fontenc}
\usepackage{mathtools}
\usepackage{graphicx}
\usepackage{tikz-cd}
\usepackage{mathabx}
\usetikzlibrary{chains}
\tikzset{
	labl/.style={anchor=south, rotate=90, inner sep=.5mm}
}

\usepackage{amsmath}
\usepackage{amsthm}
\usepackage{scalerel}
\usepackage{stackengine,wasysym}
\newcommand\reallywidetilde[1]{\ThisStyle{%
		\setbox0=\hbox{$\SavedStyle#1$}%
		\stackengine{-.1\LMpt}{$\SavedStyle#1$}{%
			\stretchto{\scaleto{\SavedStyle\mkern.2mu\AC}{.5150\wd0}}{.6\ht0}%
		}{O}{c}{F}{T}{S}%
	}}

	\usepackage[all]{xy}
	\usepackage{hyperref}
	\usepackage{comment}
	\usepackage{csquotes}
	\usepackage{dsfont}
	\usepackage{array}   
	\newcolumntype{L}{>{$}l<{$}} 
	\newcolumntype{C}{>{$}p{2 cm}<{$}} 
	\newcolumntype{D}{>{$}p{1 cm}<{$}} 
	\newcolumntype{E}{>{$}p{5 cm}<{$}} 

	\newtheorem{thm}{Theorem}[section]
	
	\newtheorem{prop}[thm]{Proposition}
	
	\newtheorem{lem}[thm]{Lemma}
	\newtheorem{cor}[thm]{Corollary}
	
	\numberwithin{equation}{section}

	\theoremstyle{definition}
	
	\newtheorem{remark}[thm]{Remark}
	
	\newtheorem{ex}[thm]{Example}

	\newcommand{\Mot}{{\rm Mot}}

	\newcommand{\one}{\mathds{1}}
	\newcommand{\hh}{\mathfrak{h}}
	\newcommand{\ttt}{\mathfrak{t}}

	\newcommand{\Bl}{{\rm Bl}}
	
	\newcommand{\pr}{{\rm pr}}
	
	\newcommand{\pt}{{\rm pt}}
	\newcommand{\tr}{{\rm tr}}
	\newcommand{\sExt}{\cal{E}xt}
	\newcommand{\sE}{\cal{E}xt_\pi ^!}
	\newcommand{\ev}{{\rm ev}}

	\newcommand{\Db}{{\rm D}^{\rm b}}

	\newcommand{\Br}{{\rm Br}}
	\newcommand{\CH}{{\rm CH}}

	\newcommand{\ord}{{\rm ord}}

	\newcommand{\rk}{{\rm rk}}
	\newcommand{\coh}{{\rm Coh}}

	\newcommand{\Hom}{{\rm Hom}}

	\newcommand{\id}{{\rm id}}
	\newcommand{\ch}{{\rm ch}}
	\newcommand{\td}{{\rm td}}

	\newcommand{\R}{\mathbf{R}}

	\newcommand{\sHom}{\cal{H}om}
	\newcommand{\K}{{\rm K}}

	\newcommand{\KE}{{\rm Ker}}

	\newcommand{\cal}{\mathcal}
	\newcommand{\ka}{{\cal A}}
	
	\newcommand{\kc}{{\cal C}}
	
	\newcommand{\ke}{{\cal E}}
	\newcommand{\kf}{{\cal F}}
	\newcommand{\kg}{{\cal G}}
	
	\newcommand{\ki}{{\cal I}}

	\newcommand{\ko}{{\cal O}}

	\newcommand{\ks}{{\cal S}}
	\newcommand{\kx}{{\cal X}}

	\newcommand{\NN}{\mathbb{N}}
	
	\newcommand{\LL}{\mathbb{L}}
	\newcommand{\ZZ}{\mathbb{Z}}
	\newcommand{\QQ}{\mathbb{Q}}
	
	\newcommand{\CC}{\mathbb{C}}

	\newcommand{\PP}{\mathbb{P}}

	\renewcommand{\to}{\xymatrix@1@=15pt{\ar[r]&}}
	\newcommand{\lto}{\xymatrix@1@=15pt{&\ar[l]}}
	\renewcommand{\rightarrow}{\xymatrix@1@=15pt{\ar[r]&}}
	\renewcommand{\mapsto}{\xymatrix@1@=15pt{\ar@{|->}[r]&}}
	\newcommand{\mapslto}{\xymatrix@1@=15pt{&\ar@{|->}[l]&}}
	\renewcommand{\twoheadrightarrow}{\xymatrix@1@=18pt{\ar@{->>}[r]&}}
	\renewcommand{\hookrightarrow}{\xymatrix@1@=15pt{\ar@{^(->}[r]&}}
	\newcommand{\hook}{\xymatrix@1@=15pt{\ar@{^(->}[r]&}}
	\newcommand{\congpf}{\xymatrix@1@=15pt{\ar[r]^-\sim&}}
	\renewcommand{\cong}{\simeq}
	\makeatletter
	\def\blfootnote{\xdef\@thefnmark{}\@footnotetext}
	\makeatother
	
	\AtBeginDocument{%
		\def\MR#1{}
	} 
	
\begin{document}
	\title[Motives of moduli spaces and of special cubic fourfolds]{Motives of moduli spaces on K3 surfaces and of special cubic fourfolds}
	
	\author{Tim-Henrik Bülles}
	\address{Mathematisches Institut,
		Universit{\"a}t Bonn, Endenicher Allee 60, 53115 Bonn, Germany}
	\email{buelles@math.uni-bonn.de}
	\begin{abstract} \noindent
		For any smooth projective moduli space $M$ of Gieseker stable sheaves on a complex projective K3 surface (or an abelian surface) S, we prove that the Chow motive $\hh(M)$ becomes a direct summand of a motive $\bigoplus \hh(S^{k_i})(n_i)$ with $k_i\leq \dim(M)$. The result implies that finite dimensionality of $\hh(M)$ follows from finite dimensionality of $\hh(S)$. The technique also applies to moduli spaces of twisted sheaves and to moduli spaces of stable objects in $\Db(S,\alpha)$ for a Brauer class $\alpha\in\Br(S)$. In a similar vein, we investigate the relation between the Chow motives of a K3 surface $S$ and a cubic fourfold $X$ when there exists an isometry $\widetilde H(S,\alpha,\ZZ) \cong \widetilde H(\ka_X,\ZZ)$. In this case, we prove that there is an isomorphism of transcendental Chow motives $\ttt(S)(1) \cong \ttt(X)$.
		\vspace{-2mm}
	\end{abstract}
	\maketitle 
\section*{Introduction}
\blfootnote{The author is partially supported by SFB/TR 45 `Periods, Moduli Spaces and Arithmetic of Algebraic Varieties' of the DFG (German Research Foundation).}

Given a moduli space $M$ of stable sheaves on a K3 surface $S$, one expects that certain invariants of $M$ are determined by the geometry of $S$. We will study the relation between the Chow groups and motives of $M$ and $S$. The analogous question for moduli spaces of stable vector bundles on a curve has been settled by del Ba\~no \cite{dB}. He showed that the Chow motive of the moduli space is contained in the full pseudo-abelian tensor subcategory generated by the motive of the curve and the Lefschetz motive. For surfaces, a natural notion of stability for sheaves is provided by Gieseker stability. More generally, we will consider stability for $\alpha$-twisted sheaves with $\alpha\in \Br(S)$ a Brauer class. The case of a moduli space of Gieseker stable sheaves corresponds to the trivial Brauer class~$\alpha=1$. The first main result of this paper is the following:
\begin{thm} \label{thm:main1}
	Let $S$ be a complex projective K3 surface or an abelian surface and $\alpha \in \Br(S)$. Assume that $M$ is one of the following:
	\begin{itemize}
		\item a smooth projective moduli space of Gieseker stable $\alpha$-twisted sheaves or
		\item a smooth projective moduli space of $\sigma$-stable objects in $\Db(S,\alpha)$, where $\sigma$ is a generic stability condition.
	\end{itemize}
	Then the Chow motive $\hh(M)$ of $M$ is a direct summand of a motive $\bigoplus \hh(S^{k_i})(n_i)$ for some $1\leq k_i \leq \dim M$, $n_i\in \ZZ$.
\end{thm}

As a direct consequence, finite dimensionality of the motive of $S$ implies the same for $M$:
\begin{cor}
	Let $S$ and $M$ be as above. If $\hh(S)$ is finite dimensional, then $\hh(M)$ is finite dimensional as well.\qed
\end{cor}

Although finite dimensionality is expected for all motives of smooth projective varieties, only a few families of K3 surfaces with finite dimensional motives are known. Even fewer examples are known in higher dimension; one example is provided by the Hilbert scheme $S^{[n]}$ of a K3 surface $S$ with finite dimensional motive, see \cite{dCM}.

The second half of this paper has a similar flavour; we investigate the relation between K3 surfaces and cubic fourfolds on the level of algebraic cycles. Recall that cubic fourfolds admitting a labelling of discriminant $d$ form a divisor $\kc_d \subseteq \kc$ inside the moduli space of smooth complex cubic fourfolds (see Section~\ref{subsec:SpecialCubics} for a brief review of the relevant notions). For a cubic fourfold $X$, we denote by $\ka_X\subseteq \Db(X)$ the Kuznetsov component of the derived category \cite{Kuz}.  We prove the following result:
\begin{thm} \label{thm:main2}
	Let $X\in \kc_d$ be a special cubic fourfold. Assume that there exist a K3 surface~$S$, a Brauer class $\alpha \in \Br(S)$ and a Hodge isometry $\widetilde H(S,\alpha, \ZZ) \cong \widetilde H(\ka_X, \ZZ)$. Then there is a cycle $Z \in \CH^3(S\times X)$ inducing an isomorphism of Chow groups $\CH_0(S)_{\hom} \congpf \CH_1(X)_{\hom}$ and transcendental motives $\ttt(S)(1) \cong \ttt(X)$. Furthermore, $\hh(X)\cong \one \oplus \hh(S)(1) \oplus \LL^2 \oplus \LL^4$ and, therefore, $\hh(S)$ is finite dimensional if and only if $\hh(X)$ is finite dimensional.
\end{thm}
Recall that a (twisted) K3 surface and a Hodge isometry $\widetilde H(S,\alpha, \ZZ) \cong \widetilde H(\ka_X, \ZZ)$ as above exist if and only if $d$ satisfies the numerical condition ($\ast$$\ast'$). 

The two results fit into the following picture. For a variety $X$ we denote by $\Mot(X)$ the full pseudo-abelian tensor subcategory of motives generated by $\hh(X)$ and the Lefschetz motive $\LL$. Let now $X$ be a cubic fourfold and $F$ its Fano variety of lines, which is a hyperkähler variety of dimension four. It is known that the motive of $F$ is contained in $\Mot(X)$ (we say that $\hh(F)$ is motivated by $\hh(X)$ following Arapura \cite{Arapura}). Indeed, Laterveer~\cite{LatFano} proved a formula for Chow motives, which is similar to the result obtained by Galkin--Shinder~\cite{GalSh} in the Grothendieck ring of varieties:
\[ \hh(F)(2) \oplus \bigoplus_{i=0} ^4 \hh(X)(i) \cong \hh(X^{[2]}). \]
Since the Hilbert scheme $X^{[2]}$ can be described as a blow-up of the symmetric product $X^{(2)}$ along the diagonal, its motive is motivated by $\hh(X)$. In Section~\ref{subsec:MotiveFano} we will argue that $\hh(X)$ is also motivated by $\hh(F)$, see also \cite[Thm.\ 4.5]{BolPed}:

\begin{cor} \label{cor:main3}
	Let $X$ be a cubic fourfold and $F$ its Fano variety of lines. The full pseudo-abelian tensor categories of motives generated by the Lefschetz motive and $\hh(X)$ and $\hh(F)$ resp., agree: 
	\[\Mot(X) = \Mot(F). \]
	In particular, $\hh(X)$ is finite dimensional if and only if $\hh(F)$ is finite dimensional.
\end{cor}

To compare this result with Theorem \ref{thm:main1}, assume that $X$ is a special cubic fourfold satisfying condition ($\ast$$\ast'$), which is equivalent to the Fano variety $F$ being birational to a moduli space $M$ of stable twisted sheaves on a K3 surface $S$, cf.\ \cite[Prop.\ 4.1]{HuyK3cat}. In this case, all of the following categories of motives agree:
\[\Mot(S) = \Mot(M) = \Mot(F) = \Mot(X).\]
Indeed, we know that birational hyperkähler varieties have isomorphic Chow motives, see Proposition~\ref{prop:Ulrike}. This induces the middle equality. It follows from Theorem~\ref{thm:main2} that $\Mot(S)$ and $\Mot(X)$ coincide. For an arbitrary complex projective K3 surface and a moduli space $M$ as in Theorem~\ref{thm:main1} we have at least an inclusion $\Mot(M) \subseteq \Mot(S)$ which we expect to be an equality as well, see Remark \ref{rem:RF} for some comments.
\subsection*{Acknowledgements} I am grateful to Daniel Huybrechts for invaluable suggestions and explanations. This work has benefited from many discussions with Thorsten Beckmann, whom I wish to thank. Finally, many thanks to Axel Kölschbach and Andrey Soldatenkov for helping to improve the exposition. This work is part of the author's Master's thesis.
\subsection*{Notations and Conventions}
We will work over the complex numbers unless otherwise stated. The bounded derived category of coherent sheaves on a smooth projective variety $X$ is denoted by $\Db(X)$. Throughout, all motives are meant to be Chow motives with rational coefficients, see Section~\ref{sec:Preliminaries}.
\section{Preliminaries} \label{sec:Preliminaries}

We briefly review the main facts about Chow motives of K3 surfaces and cubic fourfolds. The objects of the category $\Mot_\CC$ of Chow motives are triples $(X,p,m)$, with $X$ a smooth projective variety over $\CC$, $p\in \CH^{\dim(X)}(X\times X)_\QQ$ a projector (with respect to convolution) and $m$ an integer. Morphisms are defined by 
\[\Hom ((X,p,m), (Y,q,n)) = q \circ \CH^{\dim(X)+n-m} (X \times Y)_\QQ \circ p. \]
The motive of a smooth projective variety $X$ is defined as $\hh(X)=(X,[\Delta_X],0)$. We denote the motive of a point by $\one$ and the Lefschetz motive by $\LL$. Recall that for a K3 surface $S$ there is a decomposition (see e.g.\ \cite[Ch.\ 6.3]{PureMotives}):
\[ \hh(S) \cong \one \oplus \LL^{\oplus\rho(S)} \oplus \ttt(S) \oplus \LL^2. \]
The only mysterious part is the {\it transcendental motive} $\ttt(S)=(S,\pi_S^{2,\tr},0)$. The motive of a cubic fourfold $X$ splits similarly (cf.\ \cite[Sec.\ 4]{BolPed}):
\[ \hh(X) \cong \one \oplus \LL \oplus (\LL^2)^{\oplus \rho_2} \oplus \ttt(X) \oplus \LL^3 \oplus \LL^4,\]
where $\rho_2=\dim H^{2,2}(X,\QQ)$. Again, the only part which remains unclear is the transcendental motive $\ttt(X)=(X,\pi_X^{4,\tr},0)$. The above decompositions are so called {\it refined Chow--Künneth decompositions}, see \cite[Ch.\ 6.1]{PureMotives}. The Chow and cohomology groups of the transcendental motives are given by:
\begin{align*}
&H^*(\ttt(S))=H^2(\ttt(S)) = T(S)_\QQ &\text{and} \qquad &\CH^*(\ttt(S))=\CH^2(\ttt(S)) = \CH_0(S)_{\hom,\QQ}, \\
&H^*(\ttt(X))=H^4(\ttt(X)) = T(X)_\QQ &\text{and} \qquad &\CH^*(\ttt(X))=\CH^3(\ttt(X)) = \CH_1(X)_{\hom,\QQ}, \end{align*}
where $T(S)$ and $T(X)$ are the transcendental lattices.

\begin{remark}
	One can also consider the following (coarser) decomposition of the motive of a cubic fourfold $X$, which will be used in the proof of Theorem~\ref{thm:main2}. Let $h\in \CH^1(X)$ be the class of a hyperplane section and $\pt$ the class of any closed point. Define the primitive projector $ \pi_X^{\pr}=[\Delta_X] - [\pt\times X] - \frac{1}{3}[h^3 \times h] -\frac{1}{3}[h^2 \times h^2] - \frac{1}{3}[h \times h^3] - [X \times \pt]$ and the {\it primitive motive} $\hh^{\pr}(X) = (X,\pi_X^{\pr},0)$. There is a decomposition:
	\[ \hh(X) \cong \one \oplus \LL \oplus \LL^2 \oplus \hh^{\pr}(X) \oplus \LL^3 \oplus \LL^4. \]
\end{remark}

Recall the notion of a surjective morphism of motives $f\colon M \to N$. It means that the induced map $\CH^*(M\otimes \hh(Z)) \to \CH^*(N\otimes \hh(Z))$ is surjective for all smooth projective varieties $Z$, cf.\ \cite[Sec.\ 5.4]{PureMotives}. Equivalently, $f$ admits a right inverse and $N$ becomes a direct summand of $M$, see \cite[Ex.\ 2.3.(vii), Lem.\ 5.4.3]{PureMotives}. It is well known (cf.\ \cite[Lem.\ 3.2]{Vial}, \cite[Lem.\ 4.3]{Pedrini}) that it suffices to check surjectivity of $\CH^i(M_K) \to \CH^i(N_K)$ for all function fields:
\begin{lem} \label{lem:surj}
	Let $M=(X,p,m)$, $N=(Y,q,n) \in \Mot_\CC$ and $f\in \Hom(M,N)$ a morphism of motives. Assume that $(f_K)_* \colon \CH^i(M_K) \to \CH^i(N_K)$ is surjective for all finitely generated field extensions $\CC \subseteq K$. Then $f$ is surjective.
\end{lem}
\begin{proof}
	Let $Z$ be any variety over $\CC$. The proof proceeds by induction on the dimension of $Z$, the case of dimension zero being trivial. Let $K$ be the function field of $Z$ and $\gamma\in \CH^i(N\otimes \hh(Z))$. We write $\gamma | _{N_K}$ for the pullback of $\gamma$ to $N_K$. By assumption, there exists $\delta \in \CH^i(M_K)$ such that $(f_K)_*\delta  = \gamma | _{ N_K}$. Denote by $\bar \delta$ the closure of $\delta$ in $X\times Z$. Then $\gamma - (f_Z)_*\bar{\delta}$ is supported on $Y\times Z'$ for some closed proper subvariety $Z'\subseteq Z$ and we conclude by induction.
\end{proof}

In Section~\ref{sec:MotivesCubics} we will also include some comments on the notion of finite dimensionality in the sense of Kimura and O'Sullivan, see e.g.\ \cite[Ch.\ 4]{PureMotives}. The following key result is due to Kimura:

\begin{prop}[Kimura \cite{Kimura}] \label{prop:Kimura}
	Let $M \to N$ be a surjective morphism of motives. If $M$ is finite dimensional, then $N$ is finite dimensional.  If $M\cong M_1\oplus M_2$, then $M_1$ and $M_2$ are finite dimensional if and only if $M$ is finite dimensional.
	Moreover, if $X \to Y$ is a dominant morphism of smooth projective varieties and $\hh(X)$ is finite dimensional, then so is $\hh(Y)$.\qed
\end{prop} 

To conclude this section, observe that the Chow motive of a hyperkähler variety is in fact a birational invariant. Indeed, for two birational hyperkähler varieties $X$ and $X'$ one can always find families $\kx$ and $\kx'$ over a smooth quasi-projective curve $C$, which are isomorphic away from a point $0\in C$ with central fibres $X=\kx_0$ resp.\ $X'=\kx'_0$ (cf.\ \cite[Thm.\ 10.12]{HuyCHM}, \cite[Prop.\ 2.1]{Ulrike}). This can be used to show that their Chow rings $\CH^*(X)$ and $\CH^*(X')$ are isomorphic~\cite[Thm.\ 3.2]{Ulrike}. The same proof also shows that their Chow motives are isomorphic, see also \cite[Sec.\ 1.6]{ShenVial}:
\begin{prop} \label{prop:Ulrike}
	Let $X$ and $X'$ be birational hyperkähler varieties. There is an isomorphism of Chow motives \[\pushQED{\qed} \hh(X) \cong \hh(X').\qedhere \popQED\]
\end{prop}
Our result therefore also applies to any hyperkähler variety which is birational to a moduli space as in Theorem~\ref{thm:main1}. 
\section{Motives of moduli spaces of stable sheaves} \label{sec:MotivesHyper}
\subsection{Moduli spaces of stable sheaves on a K3 surface}\label{subsec:MotivesModuli}
This section contains the proof of Theorem~\ref{thm:main1}. Let $S$ be a projective (twisted) K3 surface or an abelian surface. Assume that $M$ is a smooth projective moduli space of stable (twisted) sheaves on $S$. See Remark~\ref{rem:MotiveModuli} for comments on the case of a moduli space of $\sigma$-stable objects.
\begin{proof}[Proof of Theorem~\ref{thm:main1}]
	Let $E$ be a quasi-universal sheaf on $M\times S$ and $F$ its transpose on $S \times M$. We use the following notation for the projections:
	\begin{center} \begin{tikzcd}
			&M\times S \times M \ar[dl,"\pi_{12}",swap] \ar[d,"\pi"] \ar[dr,"\pi_{23}"]& \\
			M\times S &M\times M& S\times M 
		\end{tikzcd} \end{center} and $\ke=\pi_{12}^*(E)$, $\kf=\pi_{23}^*(F)$ for the pullbacks.
		Consider the relative Ext sheaves $\sExt_\pi^i(\ke,\kf) = \R^i(\pi_* \circ \sHom ) (\ke,\kf)$ and define 
		\[ [\sE]= \sum (-1)^i [\sExt_\pi^i(\ke,\kf)]  \in \K(M\times M). \]
		Note that in our case only $\sExt_\pi^1(\ke,\kf)$ and $\sExt_\pi^2(\ke,\kf)$ are non-zero. A computation of the Chern classes due to Markman~\cite[Thm.\ 1]{Markman} yields \[ \label{eq:1} c_m(-[\sE]) = [\Delta_M] \in \CH^m(M\times M) \tag{1}, \] where $m$ is the dimension of $M$. In fact, Lemma 4 of loc.\ cit.\ also applies to moduli spaces of stable twisted sheaves.
		\smallskip
		
		Consider the Chow groups $\CH^*(M\times M)_{\QQ}$ as a unital ring with convolution of cycles and unit given by the diagonal. Define the following two-sided ideal generated by correspondences which factor through some power of $S$:
		\[ I = \langle  \beta \circ \alpha \mid \alpha \in \CH^*(M\times S^k)_{\QQ}, \beta \in \CH^*(S^k\times M)_{\QQ}, k \geq 1 \rangle \subseteq \CH^*(M\times M)_{\QQ}. \] Note that $I$ is closed under intersection products. Indeed, let $\alpha \in \CH^*(M\times S^k)_{\QQ}$, $\beta \in \CH^*(S^k\times M)_{\QQ}$, $\alpha' \in \CH^*(M\times S^{k'})_{\QQ}$, $\beta' \in \CH^*(S^{k'} \times M)_{\QQ}$ and denote by $\tau$ the involution of $M\times M\times M\times M$ interchanging the middle two factors:
		\begin{align*}(\beta \circ \alpha) \cdot (\beta'\circ \alpha')&= [\hskip 0.12em \ltrans \, \Gamma_{\Delta_{M\times M}}]_*(\beta \circ \alpha \times \beta'\circ \alpha') = [\hskip 0.12em \ltrans \, \Gamma_{\Delta_{M\times M}}]_*\circ \tau_*(\beta\times \beta' \circ \alpha\times \alpha')\\ &=[\hskip 0.12em \ltrans \,\Gamma_{\tau\circ \Delta_{M\times M}}]_*(\beta\times \beta' \circ \alpha\times\alpha')=\left([\hskip 0.12em \ltrans \,\Gamma_{\Delta_M}] \times [\hskip 0.12em \ltrans \,\Gamma_{\Delta_M} ]\right)_*(\beta\times \beta' \circ \alpha\times\alpha')\\ &= \left([\hskip 0.12em \ltrans \, \Gamma_{\Delta_M}]\circ\beta\times \beta'\right) \circ \left( \alpha\times\alpha' \circ [\Gamma_{\Delta_M}] \right). \end{align*}
		The last equality follows from Lieberman's Lemma, cf.\ \cite[Prop.\ 2.1.3]{PureMotives}. We obtain a correspondence which factors through $S^{k+k'}$, so it is contained in $I$. We will conclude by showing that the class of the diagonal is contained in $I$.\smallskip
		
		A Grothendieck--Riemann--Roch computation gives:
		\begin{align*} \label{eq:2}
		\ch\Big( - [\sE] \Big) &= - \ch \Big( \pi_! [\R\sHom(\ke,\kf)] \Big) = - \pi_* \Big( \ch [\R\sHom(\ke, \kf)] \cdot  \pi_2^* \td(S) \Big) \\ &= - \pi_* \Big( \pi_{12}^* \ch(E^{\vee})\cdot \pi_{23}^* \ch(F)\cdot \pi_2^* \td(S) \Big), \tag{2} \end{align*}
		where $E^{\vee}=\R\sHom(E,\ko_{M\times S})$ denotes the derived dual of $E$ and $\pi_2$ is the projection to $S$. Let $\alpha = \oplus \alpha^i= \ch(E^{\vee})\cdot  \pi_2^* \sqrt{\td(S)}$, $\beta= \oplus\beta^i =\ch(F)\cdot  \pi_2^* \sqrt{\td(S)}$ and $n\in \NN$. Considering only the codimension $n$ part of \eqref{eq:2} we find that the $n$-th Chern character is contained in $I$:
		\[ \ch_n \Big(-[\sE] \Big)= - \sum_{i+j=n+2} \pi_*(\pi_{12}^*\alpha^i \cdot \pi_{23}^*\beta^j) \in I.\] 
		The codimension $n$ part of the Chern character is given as a sum  $\displaystyle{\frac{(-1)^{n-1}}{(n-1)!}c_n + p}$, where $p$ is a polynomial in the Chern classes of degree less than $n$. Note that $c_1=\ch_1$ is contained in $I$ and, therefore, also $c_2=\frac{1}{2} c_1^2-\ch_2\in I$. It follows iteratively that $c_n\in I$ for all $n$ and therefore $[\Delta_M] \in I$ by \eqref{eq:1}. Thus, there are cycles $\gamma_i \in \CH^{e_i}(M\times S^{k_i})_{\QQ}$, $\delta_i \in \CH^{d_i}(S^{k_i} \times M)_{\QQ}$, for some $k_i\in \NN$, such that
		\[\label{eq:3} [\Delta_M] = \sum \delta_i \circ \gamma_i \in \CH^m(M\times M)_{\QQ}. \tag{3}\]
		Let $\delta = \bigoplus \delta_i$ viewed as a morphism of motives $\bigoplus \hh(S^{k_i})(n_i) \to \hh(M)$ with $n_i=d_i-2k_i$. Equation \eqref{eq:3} asserts that $\gamma = \bigoplus \gamma_i$ defines a right inverse for $\delta$, i.e.\ the following composition is the identity:
		\begin{center}\begin{tikzcd} \hh(M) \ar[r, "\gamma"] & \bigoplus \hh(S^{k_i})(n_i) \ar[r,"\delta"] & \hh(M). \end{tikzcd} \end{center}
		Hence, $\hh(M)$ is a direct summand of $\bigoplus \hh(S^{k_i})(n_i)$.
		
		Moreover, we obtain a bound for the exponents $k_i$. Consider the filtration $I_k$ of $I$ generated by correspondences which factor through $S^l$ with $l \leq k$. With the above notation we have $\ch_n \in I_1$ for all $n$ and $I_k \cdot I_{k'} \subseteq I_{k+k'}$. Thus $k_i\leq m=\dim M$ for all $i$.
	\end{proof}
	\begin{remark} \label{rem:MotiveModuli}
		The above argument also works for smooth projective moduli spaces $M_\sigma(v)$ of $\sigma$-stable objects for a generic stability condition $\sigma$. It was observed in \cite{MaZh} that Markman's computation of the Chern class can be carried out similarly in this case.
	\end{remark}
	\begin{cor}
		Let $S$ and $M$ be as above. If $\hh(S)$ is finite dimensional, then $\hh(M)$ is finite dimensional as well.\qed
	\end{cor}
	\begin{remark} \label{rem:RF}
		We expect also that $\hh(S)$ is motivated by $\hh(M)$ (see the introduction). This holds for example in the case of a Hilbert scheme. For fine moduli spaces it would follow from a conjecture of Addington \cite{AddRF}: A universal sheaf induces a Fourier--Mukai transform $F\colon \Db(S) \to \Db(M)$ with right adjoint $R$. Addington conjectured that the composition of $F$ and $R$ splits as follows:
		\[ R \circ F \cong \id \oplus \id[-2] \oplus \ldots \oplus \id[-2n+2]. \] 
		If $v$ and $w$ are the Mukai vectors of the Fourier--Mukai kernels, we obtain:
		\[  [\Delta_S] = \frac{1}{n} v \circ w \in \CH^2(S \times S)_\QQ. \] 
		It follows as above that $\hh(S)$ is a direct summand of $\bigoplus \hh(M)(n_i)$ for some $n_i \in \ZZ$.
	\end{remark}

	\subsection{The Fano variety of lines} \label{subsec:MotiveFano}
	We provide a short proof of Corollary~\ref{cor:main3}. Let $X$ be a cubic fourfold and $F$ its Fano variety of lines. The Chow groups and motive of $F$ were investigated in detail by Shen and Vial \cite{ShenVial}. They studied Fourier transforms inducing a (particularly interesting) decomposition of the Chow ring, similar to the case of an abelian variety. The relation between the Chow groups of $F$ and $X$ given via the universal line (viewed as a correspondence) has been elucidated as well. We refrain from going into the details and recommend loc.\ cit.\ for further reading.
	
	\begin{prop}
		Let $X$ be a cubic fourfold and $F$ its Fano variety of lines. Then the transcendental motive $\ttt(X)$ is a direct summand of $\hh(F)(-1)$. In particular, the motive of $X$ is contained in $\Mot(F)$.
	\end{prop}
	
	\begin{proof}
		The universal line $L\in \CH^3(F\times X)$ induces a morphism $f$ of motives:
		\begin{center} \begin{tikzcd}
				\hh(F)(-1) \ar[r, "L"] & \hh(X) \ar[r, "\pi_X^{4,\tr}"] & \ttt(X). \end{tikzcd}
		\end{center}
		Let $K$ be any finitely generated field extension of $\CC$. The only non-trivial rational Chow group of $\ttt(X_K)$ is $\CH^3(\ttt(X_K))\cong \CH_1(X_K)_{\hom,\QQ}$. Indeed, choose an embedding of $K$ into the complex numbers and denote by $Y$ the base change of $X_K$ to $\CC$, which is a smooth complex cubic fourfold. It is well known that the base change map $\CH^i(\ttt(X_K)) \to \CH^i(\ttt(Y))$ induced by a field extension is injective up to torsion, see e.g.\ \cite[Lem.\ 1A.3]{Bloch}. Now use that $\CH^i(\ttt(Y))$ vanishes for $i \neq 3$.
		The Chow group of one-cycles is universally generated by lines (cf.\ \cite{Mingmin2}) and the assertion thus follows from Lemma~\ref{lem:surj}.
	\end{proof}

\section{Motives of special cubic fourfolds} \label{sec:MotivesCubics}
\subsection{Special cubic fourfolds}\label{subsec:SpecialCubics} Recall that cubic fourfolds admitting a labelling of discriminant $d$ form a divisor $\kc_d \subseteq \kc$ inside the moduli space of smooth complex cubic fourfolds, see \cite{HassettThesis}. The existence of an associated K3 surface (in a suitable sense) can be characterized solely in terms of $d$. The following numerical conditions have been introduced over the past years (we use the notation of Addington~\cite{Add}):

\begin{align*}
&\exists a,n \in \ZZ : a^2d=2n^2+2n+2, &&(\ast {\ast} \ast) \\
&\exists n \in \ZZ : d \mid 2n^2+2n+2, &&(\ast \ast ) \\
&\exists k,d_0 \in \ZZ : \text{$d_0$ satisfies }(\ast \ast ) \text{ and }  d=k^2d_0. &&(\ast \ast' )
\end{align*}

There are (strict) inclusions of subsets inside the moduli space $\kc$ of cubic fourfolds:
\[ \bigcup\limits_{(\ast {\ast} \ast)}\kc_d\quad \subseteq \quad \bigcup\limits_{(\ast \ast)}\kc_d \quad \subseteq \quad \bigcup\limits_{(\ast \ast')}\kc_d. \]
A cubic fourfold admits a labelling of discriminant $d$ satisfying $(\ast \ast' )$ if and only if there exist a K3 surface $S$, a Brauer class $\alpha\in\Br(S)$ and a Hodge isometry $\widetilde H(S,\alpha,\ZZ) \cong \widetilde H(\ka_X,\ZZ)$ \cite[Thm.\ 1.3]{HuyK3cat}. In this case, we prove that there is an isomorphism of Chow motives $\ttt(S)(1)\cong \ttt(X)$. This generalizes work of Bolognesi, Pedrini \cite{BolPed}, and Laterveer \cite{Lat17}. In \cite{BolPed}, the authors obtained such an isomorphism in the case when $F(X) \cong S^{[2]}$. Injectivity has been proven in \cite{Lat17} for cubic fourfolds invariant under a certain involution. Both cases are instances of Theorem~\ref{thm:main2}, see the comments in Section~\ref{subsec:examples}. We start with a well known fact:

\begin{lem} \label{divisible}
	Let $S$ be a projective K3 surface and $X$ a cubic fourfold. Then $\CH_0(S)_{\hom}$ and $\CH_1(X)_{\hom}$ are divisible and torsion-free.
\end{lem}

\begin{proof}
	Divisibility of $\CH_0(S)_{\hom}$ is well known and follows easily by constructing a curve through any two given points and using the Jacobian of the normalization. The theorem of Rojtman \cite{Rojtman} implies that this group is torsion-free. 
	Let $F$ be the Fano variety of lines in $X$. It is a hyperkähler variety, so its first Betti number vanishes and it follows as above that  $\CH_0(F)_{\hom}$ is divisible and torsion-free. The universal line $L$ induces a surjection \begin{center}
		\begin{tikzcd} \CH_0(F)_{\hom} \ar[r,"L_*"] &\CH_1(X)_{\hom}, \end{tikzcd} \end{center}
	hence the assertion follows from the divisibility of $\KE(L_*)$ which was proven by Shen and Vial~\cite[Thm.\ 20.5, Lem.\ 20.6]{ShenVial}.
\end{proof}
\begin{proof}[Proof of Theorem \ref{thm:main2}]
	Since $\CC$ is a universal domain, it suffices to prove the isomorphism on Chow groups. By a variant of Manin's identity principle (cf. \cite[Lem.\ 1]{GG}, \cite[Lem.\ 3.2]{Vial} or \cite[Lem.\ 4.3]{Pedrini}) this implies $\ttt(S)(1) \cong \ttt(X)$.
	The results of Addington--Thomas \cite{AT} and Huybrechts \cite{HuyK3cat} imply that there is an exact equivalence $\Db(S) \simeq \ka_X$ (resp.\ $\Db(S,\alpha) \simeq \ka_X$) if $X\in \kc_d$ is generic\footnotemark and we consider this case first.
	\footnotetext{At the moment, an equivalence $\Db(S) \cong \ka_X$ (resp.\ $\Db(S,\alpha) \cong \ka_X$) is established only for generic $X \in \kc_d$. This gap is expected to be filled soon and would make the last step of the proof superfluous (see the upcoming work of Bayer, Lahoz, Macr{\`i}, Nuer, Perry, Stellari \cite{BLMNPS}).}
	Assume that $\alpha = 1$, i.e.\ $d$ satisfies ($\ast$$\ast$). Consider the composition $\Phi$ of an exact equivalence $\Db(S) \simeq \ka_X$ and the inclusion $\ka_X \subseteq \Db(X)$. By \cite{Orlov}, this functor is of Fourier--Mukai type, i.e.\ there is a complex $\ke \in \Db(S \times X)$, such that for all $\kg \in \Db(S)$:
	\[ \Phi(\kg) \cong p_*(\ke \otimes q^*(\kg)), \]
	where $p$ and $q$ are the projections. It follows that the left adjoint to $\Phi$ is of Fourier--Mukai type as well, say with kernel $\kf$. 
	Let $v= \ch(\ke) \cdot \sqrt{\td(S \times X)}$ (resp.\ $w$) be the Mukai vector of $\ke$ (resp.\ $\kf$). It is an algebraic cycle with $\QQ$-coefficients on $S \times X$ which needs not be of pure dimension. Denote by $v^i$ (resp.\ $w^i$) its codimension $i$ part. Since $\Phi$ is fully faithful, the convolution $w \circ v$ is rationally equivalent to the class of the diagonal $[\Delta_S]$ on $S \times S$. More precisely, the following equality holds in $\CH^2(S\times S)_\QQ$:
	\begin{equation} \label{eq:4} [\Delta_S] = w^0\circ v^6 + w^1\circ v^5 + w^2\circ v^4 + w^3\circ v^3 + w^4\circ v^2 + w^5\circ v^1 + w^6\circ v^0. \tag{4} \end{equation}
	
	Recall that the homologically trivial part of the Chow groups of $S$ and $X$ are concentrated in codimension two and three, respectively. The induced action of $v$ on Chow groups is compatible with the action on cohomology. Thus, $w^3 \circ v^3$ is the only summand on the right hand side of \eqref{eq:4} acting non-trivially on $\CH_0(S)_{\hom,\QQ}$, i.e.\ the following composition is the identity:
	
	\[ \begin{tikzcd} \CH_0(S)_{\hom,\QQ} \ar[r, "v^3_*"] & \CH_1(X)_{\hom,\QQ} \ar[r, "w^3_*"] & \CH_0(S)_{\hom,\QQ}.
	\end{tikzcd} \] This proves injectivity of $v^3_*$.
	For the surjectivity consider the following diagram:
	
	\[\begin{tikzcd}
	\K(S)_\QQ \ar[d,  "v"] \ar[r,"\sim"] & \K(\ka_X)_\QQ \ar[r] \ar[dr, "\phi"] & 	\K(X)_\QQ \ar[d, "v"] \\
	\CH^*(S)_\QQ \ar[rr, "v_*"]  && \CH^*(X)_\QQ \\
	\CH_0(S)_{\hom,\QQ} \ar[rr, "{v^3_*}"] \ar[u, hook] && \CH_1(X)_{\hom,\QQ}. \ar[u, hook]
	\end{tikzcd} \]
	Commutativity of the middle diagram follows from the Grothendieck--Riemann--Roch Theorem. It suffices to show that the image of $\phi \colon \K(\ka_X)_\QQ \to \CH^*(X)_{\QQ}$ contains $\CH_1(X)_{\hom,\QQ}$. Indeed, this would imply that any $\beta\in\CH_1(X)_{\hom,\QQ}$ lifts to some $\alpha\in \CH^*(S)_{\QQ}$ such that $v_*(\alpha)=\beta$. Since the action of $v$ on cohomology is injective, $\alpha$ is homologically trivial, i.e.\ $\alpha\in \CH_0(S)_{\hom,\QQ}$. 
	
	Recall that $\CH_1(X)$ is generated by lines by a result of Paranjape \cite{Paranjape}, see also \cite[Cor.\ 4.3]{Mingmin1}. Let $i \colon \ell \subseteq X$ be the inclusion of a line and consider the associated second syzygy sheaf $\kf_\ell$ of $\ki_\ell(1)$ defined by:
	\[ \begin{tikzcd}
	0 \ar[r] & \kf_\ell \ar[r] & H^0(X, \ki_\ell(1))\otimes \ko_X \ar[r, "\ev"] & \ki_\ell(1) \ar[r] & 0.
	\end{tikzcd}\]
	Here, $\ko_X(1)$ is the induced polarization of $X \subseteq \PP^5$ and $\ev$ is the evaluation map which is surjective, cf.\ \cite[Lem.\ 5.1]{KuzMar}. A straightforward computation in loc.\ cit.\ shows that $\kf_\ell$ is contained in $\ka_X$. Next, we compute the Mukai vector of $\kf_\ell$:
	\[ v(\kf_\ell)= v(\ko_X^{\oplus 4}) - v(\ki_\ell(1)) = v(\ko_X^{\oplus 4}) - v(\ko_X(1)) + v(\ko_\ell(1)). \]
	Using the Grothendieck--Riemann--Roch Theorem one finds:
	\begin{align*} 
	v(\ko_\ell(1))&= \ch(\ko_\ell) \cdot \ch(\ko_X(1)) \cdot  \td(X)^{\frac{1}{2}}= i_* (\td(\ell)) \cdot \ch(\ko_X(1)) \cdot \td(X)^{-\frac{1}{2}} \\
	&= ([\ell] + [\pt])\cdot \ch(\ko_X(1)) \cdot \td(X)^{-\frac{1}{2}},
	\end{align*}
	where $[\pt] \in \CH_0(X)\cong \ZZ$ is the class of any closed point ($X$ is rationally connected). The Todd class of $X$ is a polynomial in the class of a hyperplane section $h = c_1(\ko_X(1))$, in fact
	\[ \td(X)= 1 + \frac{3}{2} h + \frac{5}{4} h^2 + \frac{3}{4} h^3 + \frac{1}{3}h^4.\]
	Therefore, $v(\ko_\ell(1))= [\ell] + \frac{5}{4} [\pt]$ and
	\[ \phi([\kf_\ell] - [\kf_{\ell'}]) = v(\ko_\ell(1)) - v(\ko_{\ell'}(1))= [\ell] - [\ell'], \]
	for each pair of lines $\ell$ and $\ell'$, which proves surjectivity of $\phi$ since $\CH_1(X)_{\hom,\QQ}$ is generated by cycles of this form.
	
	So far, we proved that $Z=v^3$ induces an isomorphism $\CH_0(S)_{\hom,\QQ} \congpf \CH_1(X)_{\hom,\QQ}$. As mentioned earlier, a variant of Manin's identity principle gives that $Z$ also induces an isomorphism of motives $\ttt(S)(1) \cong \ttt(X)$, which extends to an isomorphism $\hh(S)(1) \cong \LL \oplus \hh^{\pr}(X) \oplus \LL^3$. Indeed, the Picard rank $\rho$ of $S$ equals $\rho_2-1$ with $\rho_2 = \dim H^{2,2}(X,\QQ)$. Thus, there are cycles $W$, $W' \in \CH^3(S \times X)_\QQ$ such that 
	\begin{equation} \label{eq:5}\hskip 0.12em \ltrans \, W' \circ W = [\Delta_S],\quad  W \circ \hskip 0.12em \ltrans \, W' = \frac{1}{3}[h^3 \times h] + \pi_X^{\pr} + \frac{1}{3}[h \times h^3]. \tag{5}  \end{equation}
	This will be useful for the specialization argument below.\vskip0.4cm
	
	Next, assume that $d$ satisfies ($\ast$$\ast'$), i.e.\ $\Db(S, \alpha) \cong \ka_X$. The composition with the inclusion is again of Fourier--Mukai type (cf.\ \cite{TwistedFM}) and the formalism of Mukai vectors works in the twisted case as well, see \cite{EquivTwisted} for details. For $E \in \coh(S \times X, \alpha^{-1} \boxtimes 1)$ locally free and $n=\ord(\alpha)$ the order of the Brauer class, $E^{\otimes n}$ is naturally an untwisted sheaf and one defines (cf.\ \cite[Sec.\ 2.1]{MotivesIsog})
	\[v(E) = \sqrt[n]{\ch(E^{\otimes n})} \cdot \sqrt{\td(S\times X)}. \]
	The $n$-th root can be obtained formally, since $\rk(E) \neq 0$. Using a locally free resolution, this definition extends to all twisted coherent sheaves. Define the cycle $Z$ as above. The proof now works analogously, replacing $\Db(S)$ by $\Db(S,\alpha)$ and $\K(S)$ by $\K(S, \alpha)$. \vskip0.4cm
	
	Finally, we prove the assertion for any $X_0\in \kc_d$ via specialization. Let $T \subseteq \kc_d$ be a curve passing through the point corresponding to $X_0$ such that there are families of K3 surfaces (resp.\ cubic fourfolds) $\ks$ and $\kx$ over $T$ with an exact equivalence $\Db(\ks_s) \cong \ka_{\kx_s}$ over a very general point $s\in T$ and $\kx_0 \cong X_0$ for a closed point $0\in T$, see \cite{AT}. Write $S_0$ for the fibre of $\ks$ over 0.
	
	By a standard argument (see e.g.\ \cite[Lem.\ 8]{Stefan}) we may assume that $T$ is the spectrum of a complete discrete valuation ring $R\cong \CC \llbracket t \rrbracket $ with generic point $\eta$ and closed point $0$. Write $K=\CC(\!(t)\!)$ for its fraction field and $\bar{K}$ for an algebraic closure of $K$.
	
	Let $W$, $W' \in \CH^3(\ks_{\bar{\eta}} \times_{\bar{K}} \kx_{\bar{\eta}})$ be as above, such that \eqref{eq:5} holds. In fact, all cycles of  \eqref{eq:5} are defined over a finite extension $\CC(\!(t^{\frac{1}{n}})\!)$ of $K$. Replacing $R$ by $\CC\llbracket t^{\frac{1}{n}}\rrbracket$, we may assume that the cycles $W$ and $W'$ are defined over $K$.
	Recall the specialization map for Chow groups (see \cite[Ch.\ 10.1]{Fulton} for details), which is compatible with intersection product, pullback and proper pushforward. We obtain cycles $W_0$, $W_0' \in \CH^3(S_0 \times X_0)_\QQ$ such that equalities of the form \eqref{eq:5} hold. Thus, $W_0$ induces an isomorphism of motives $\hh(S_0)(1) \cong \LL \oplus \hh^{\pr}(X_0) \oplus \LL^3$. The action on Chow groups restricts to an isomorphism of homologically trivial cycles $\CH_0(S)_{\hom,\QQ} \congpf \CH_1(X)_{\hom,\QQ}$ induced by $\pi_{X_0}^{4,\tr} \circ W_0 \circ \pi_{S_0}^{2,\tr}$. In fact, $\CH_0(S)_{\hom}$ and $\CH_1(X)_{\hom}$ are both divisible and torsion-free, see Lemma~\ref{divisible}. Hence, tensoring with $\QQ$ is a bijection and we obtain an isomorphism of integral Chow groups.
\end{proof}

\begin{cor} Let $X\in \kc_d$ be a special cubic fourfold with d satisfying \emph{($\ast$$\ast'$)} and $S$ an associated (twisted) K3 surface. Then $\hh(X)$ is finite dimensional if and only if $\hh(S)$ is finite dimensional. Moreover, if $\rho_2=\dim H^{2,2}(X,\QQ) \geq 20$, then $\hh(X)$ is finite dimensional.
	
\end{cor}
\begin{proof}
	The above theorem evidently implies $\hh(X) \cong \one \oplus \hh(S)(1) \oplus \LL^2 \oplus \LL^4$. This proves the first assertion. If $\rho_2=\dim H^{2,2}(X,\QQ) \geq 20$, then the Picard rank of $S$ is at least 19 and, therefore, $S$ admits a Shioda--Inose structure, cf.\ \cite[Cor.\ 6.4]{Morrison}. The motive of an abelian variety is finite dimensional, see e.g.\ \cite[Ch.\ 4.6, Thm.\ 2.7.2]{PureMotives}. Thus, $\hh(S)$ is finite dimensional and we conclude using Proposition~\ref{prop:Kimura}.
\end{proof}
\subsection{Examples} \label{subsec:examples}
This section contains a comparison with the work of Bolognesi, Pedrini \cite{BolPed} and some applications of Theorem~\ref{thm:main2}. In each example, the relation on the level of motives between the K3 surface and the cubic fourfold becomes visible by a concrete geometric construction.

\begin{ex} [Cubic fourfolds containing a plane]\label{ex:plane}
	Consider the divisor $\kc_8 \subseteq \kc$. It corresponds exactly to the cubic fourfolds $X$ containing a plane, cf.\ \cite[Sec.\ 3]{Voisin}. In this case, there is the following standard construction:
	Let $\widetilde{X}$ be the blow-up of $X$ along a plane $P$. Projecting $X$ from $P$ onto a disjoint plane in $\PP^5$ yields a rational map which can be resolved to give a morphism $q \colon \widetilde{X} \to \PP^2$. The fiber of $q$ over a point $x\in \PP^2$ is the residual surface of the intersection $\overline{xP} \cap X$. Generically, it is a smooth quadric surface, i.e.\ isomorphic to $\PP^1 \times \PP^1$ and has two different rulings. The discriminant divisor of $q$ is a sextic curve in $\PP^2$ over which each fibre is singular with only one ruling. More precisely, let $F(\widetilde{X}/\PP^2)$ be the relative Fano variety of lines with universal line $L \subseteq F(\widetilde{X}/\PP^2) \times \widetilde{X}$. The projection $L \to \PP^2$ factors through a double cover $S\to \PP^2$ branched along a sextic curve, which is smooth for a general choice of $X$. Thus, $S$ is a K3 surface. The projection $L \to S$ is a $\PP^1$-bundle (a Brauer--Severi variety) and induces a Brauer class $\alpha \in \Br(S)$. 
	Kuznetsov showed that there is an exact equivalence $\Db(S,\alpha) \cong \ka_X$, cf.\ \cite[Thm.\ 4.3]{Kuz}. 
	\medskip
	
	It is well known that rationality of the cubic fourfold $X$ follows, if $q$ has a rational section. This holds true if there is an additional surface $W\subseteq X$ such that $\deg(W) - \langle P,W \rangle $ is odd. In this case, it was observed in \cite[Sec.\ 8]{BolPed} that the isomorphism $\ttt(S)(1) \cong \ttt(X)$ would follow from finite dimensionality of $\hh(S)$. In fact, Theorem~\ref{thm:main2} implies that the isomorphism $\ttt(S)(1) \cong \ttt(X)$ holds without any further assumptions.
\end{ex}
\begin{ex}[Cubic fourfolds with an automorphism of order three]
	Let $X$ be a cubic fourfold given by an equation of the form
	\[ f(x_0,x_1,x_2) - g(x_3,x_4,x_5) =0, \] 
	where $f$ and $g$ are homogeneous polynomials of degree three. Denote by $\zeta_3$ a primitive third root of unity. Then $X$ is invariant under the automorphism $\sigma$ of $\PP^5$ given by 
	\[[x_0 : x_1 : x_2 : x_3 : x_4 : x_5] \mapsto [x_0 : x_1 : x_2 : \zeta_3 x_3 : \zeta_3 x_4 : \zeta_3 x_5]. \] Thus, there is an induced automorphism $\sigma_F$ of the Fano variety $F(X)$, which is in fact symplectic, i.e.\ $\sigma_F | _{H^{2,0}} = \id $, see e.g.\ \cite{Fu} for a classification of polarized symplectic automorphisms of $F(X)$. Consider the cubic surfaces $Z_1= \{ f(x_0,x_1,x_2) - s^3 =0 \}$ and $Z_2=\{g(x_3,x_4,x_5)-t^3=0 \}$ in $\PP^3$ with $s$ resp.\ $t$ as additional variables. The rational map 
	\[ ([x_0:x_1:x_2:s], [x_3:x_4:x_5:t]) \mapsto [\frac{x_0}{s}:\frac{x_1}{s}:\frac{x_2}{s}:\frac{x_3}{t}:\frac{x_4}{t}:\frac{x_5}{t}] \]
	induces a degree three morphism $\reallywidetilde{Z_1 \times Z_2} \to X$ from the blow-up of $Z_1 \times Z_2$ along $E_1 \times E_2$. Here, $E_i$ is the cubic curve in $Z_i$ defined by the vanishing of $s$ resp.\ $t$, see e.g. \cite[Prop.\ 1.2]{CT}. 
	
	Note that finite dimensionality of $\hh(X)$ follows from Proposition~\ref{prop:Kimura} since rational surfaces have finite dimensional motives. Moreover, this morphism can be used to find two disjoint planes $P_1$ and $P_2$ contained in $X$; if $\ell_i \subseteq Z_i$ are lines (recall that $Z_i$ contains 27 of them) the image of the product $\ell_1 \times \ell_2$ is a plane in $X$ and certain choices of lines produce disjoint planes, cf.\ \cite[Rem.\ 2.4]{CT}. There is a birational map from $P_1 \times P_2$ to $X$ sending a pair of points $(x,y)$ to the residual point of the intersection $\overline{xy} \cap X$. The indeterminacy locus $S\subseteq P_1 \times P_2$ parametrizes lines contained in $X$ joining the two planes. It is a complete intersection of divisors of type $(1,2)$ and $(2,1)$, i.e.\ $S$ is a K3 surface, see \cite[Ex.\ 5.9]{GalSh}. Resolving the indeterminacy locus gives an isomorphism $\Bl_S(P_1 \times P_2) \congpf \Bl_{P_1 \cup P_2}(X)$ which induces $\ttt(S)(1) \cong \ttt(X)$ by comparing homologically trivial cycles. In fact, the cubic fourfold $X$ satisfies condition $(\ast {\ast} \ast )$, since the Fano variety of $X$ is birational to the Hilbert scheme $S^{[2]}$.
\end{ex}
\begin{ex} [Cubic fourfolds with an involution]
	Consider the involution $\sigma$ on $\PP^5$ given by 
	\[[x_0 : x_1 : x_2 : x_3 : x_4 : x_5] \mapsto [x_0 : x_1 : x_2 : x_3 : -x_4 : -x_5].\]
	A cubic $X$ invariant under $\sigma$ is always of the form
	\[ \{ F(x_0,x_1,x_2,x_3) + x_4^2L_1 + x_5^2L_2 + x_4x_5L_3=0 \}, \]
	where $F$ is homogeneous of degree three and the $L_i$ are linear forms in $x_0,\ldots,x_3$. Note that the fixed locus of $\sigma$ in $\PP^5$ is the union of $\PP^3 = \{x_4 = x_5 = 0\}$ and the line $\ell = \{ [0:0:0:0:x_4:x_5] \}$. Thus, the fixed locus in $X$ consists of a cubic surface $W$ and the line $\ell$. 
	
	It was shown in \cite{Fu} that $\sigma$ induces a symplectic involution on the Fano variety $F(X)$. Moreover, the fixed locus in $F(X)$ can be described explicitly. It consists of the line $\ell$, the 27 lines contained in $W$ and a K3 surface $S$. The surface $S$ parametrizes lines contained in $X$ joining $W$ and $\ell$. It is a double cover of the cubic $W$ branched along the degree 6 curve $L_3^2 - L_1L_2$.
	This suggests that $S$ is associated to $X$: The inclusion $S \subseteq F(X)$ induces an isomorphism $H^{2,0}(F(X)) \cong H^{2,0}(S)$ and an isomorphism of transcendental lattices. Composing with the incidence correspondence, we get $T(S)(-1) \cong T(X)$. It is not directly obvious that this is an isometry. An isomorphism $\ttt(S)(1) \cong \ttt(X)$ was nevertheless established by Bolognesi and Pedrini~\cite[Sec.\ 5.2]{BolPed}) building on work of Laterveer~\cite{Lat17}.
\end{ex}
\begin{ex}[Cyclic cubic fourfolds]
	Let $f(x_0,\ldots ,x_4)$ be a homogeneous polynomial of degree three, defining a smooth cubic threefold $Y\subseteq \PP^4$. A cyclic cubic fourfold is a triple cover $X\to \PP^4$ ramified along $Y$. It is a smooth cubic hypersurface $X\subseteq \PP^5$ with an equation:
	\[ f(x_0,\ldots,x_4) + x_5^3=0 \]
	and covering automorphism $\sigma\colon X \congpf X$ given by: 
	\[[x_0 : x_1 : x_2 : x_3 : x_4 : x_5] \mapsto [x_0 : x_1 : x_2 : x_3 : x_4 : \zeta_3 x_5]. \]
	
	It was shown in \cite{LatCubics} that the motive of a cyclic cubic fourfold $X$ is finite dimensional. If $X$ satisfies condition ($\ast$$\ast'$) and $S$ is an associated (twisted) K3 surface, then $\ttt(S)(1) \cong \ttt(X)$ and $\hh(S)$ is finite dimensional as well. Unfortunately, it is not clear which K3 surfaces can be associated to $X$ as above. Note that the family of cyclic cubic fourfolds contains the Fermat cubic, so in particular it has non-trivial intersection with the divisor $\kc_8$ of cubic fourfolds containing a plane. However, there exists an example of a cyclic Pfaffian cubic fourfold containing no plane, see \cite[Prop.\ 5.1]{BCS2}.
\end{ex}

\bibliography{refs}
\bibliographystyle{myamsplain}
\end{document}